\documentclass[12pt]{amsart}

\usepackage[top=1in, bottom=1in, left=1in, right=1in]{geometry}
\usepackage{times}
\usepackage{fancyhdr}

\usepackage{amssymb,amsmath,amsthm,mathtools,bbm}
\usepackage{enumitem}

\setlist[1]{itemsep=0.5em, topsep=0.5em}

\usepackage{dsfont}
\usepackage{mathrsfs}

\usepackage{color}
\definecolor{red}{rgb}{1,0,0}
\definecolor{orange}{rgb}{0.7,0.3,0}
\definecolor{blue}{rgb}{0,.3,.7}
\definecolor{green}{rgb}{0,.6,.4}
\definecolor{dblue}{rgb}{0,.3,.7}

\usepackage[colorlinks=true, linkcolor=dblue, citecolor=green, urlcolor=purple]{hyperref}   

\numberwithin{equation}{section}

\newtheorem{theorem}{Theorem}[section]
\newtheorem{lemma}[theorem]{Lemma}

\theoremstyle{remark}

\newtheorem*{rem*}{Remark}

\newtheoremstyle{claim} 
    {1em}                    
    {1em}                    
    {}                   
    {}                           
    {\bfseries}                   
    {.}                          
    {.5em}                       
    {}  
    
\theoremstyle{claim}

\newcommand{\N}{\mathbb{N}}
\newcommand{\Z}{\mathbb{Z}}

\newcommand{\R}{\mathbb{R}}
\newcommand{\C}{\mathbb{C}}

\newcommand{\CF}{\mathcal{F}}

\renewcommand{\le}{\leqslant}
\renewcommand{\ge}{\geqslant}

\newcommand{\bs}\boldsymbol{}

\newcommand{\dee}{\mathrm{d}}

\renewcommand{\tilde}{\widetilde}

\renewcommand{\phi}{\varphi}

\renewcommand{\Re}{\mathrm{Re}}

\begin{document}

\title{A logarithmic structure theorem for multiplicative functions with small partial sums}

\author{Dimitris Koukoulopoulos}
\address{D\'epartement de math\'ematiques et de statistique\\
Universit\'e de Montr\'eal\\
CP 6128 succ. Centre-Ville\\
Montr\'eal, QC H3C 3J7\\
Canada}
\email{dimitris.koukoulopoulos@dms.umontreal.ca}


\date{\today}

\maketitle

\begin{center}
	\dedicatory{\textit{ \small{Dedicated to Roger Heath-Brown on the occasion of his 75th birthday}}}
\end{center}

\begin{abstract}
Let $D\in\mathbb{N}$, let $A>D+1$, and let $Q\geqslant3$. Consider the class of multiplicative functions $f:\mathbb{N}\to\mathbb{C}$ such that $|\sum_{n\leqslant x}f(n)|\le x(\log Q)^{A-D-1}/(\log x)^A$ for all $x\geqslant Q$, and such that $|\Lambda_f|\leqslant D\Lambda$, where $\Lambda_f$ is defined via the Dirichlet convolution identity $f\log=\Lambda_f*f$ and $\Lambda$ denotes von Mangoldt's function. We prove there exist parameters $m\in\{0,1,\dots,D\}$ and $Q=Q_D\leqslant Q_{D-1}\le \cdots\leqslant  Q_m<Q_{m+1}=\infty$ such that $\sum_{p\in I} \mathrm{Re}(f(p)+j)/p=O_{A,D}(1)$ for all $j=m,m+1,\dots,D$ and all compact intervals $I\subset[Q_{j+1},Q_j)$. Moreover, when $|\sum_{n\leqslant  x}f(n)|\le x^{1-1/\log Q}/(\log x)^{D+1}$ for all $x\geqslant Q$, we relate the parameters $m$ and $Q_j$ to the location of zeroes of the Dirichlet series $\sum_{n\geqslant1} f(n)/n^s$ in the ball $B(1,1/\log Q)$. These results generalize work of the author when $D=1$. Their proof builds on earlier work of the author with Soundararajan, and of Sachpazis.
\end{abstract}


\section{Introduction}

\subsection{A heuristic argument}\label{sec:heuristic} Let us begin by describing a classical idea about zeroes of $L$-functions. 

We consider a multiplicative function $f:\N\to\C$ and its Dirichlet series
\[
L(s,f)  =\sum_{n=1}^\infty \frac{f(n)}{n^s} .
\]
We further assume that this is an $L$-function in the sense of Selberg \cite[Chapter 5]{IK} with no pole at $s=1$ and with conductor $q$. We then expect that
\[
\sum_{p\le x}f(p)\log p \approx - \sum_\rho \frac{x^\rho}{\rho} 
\] 
for $x\ge q^C$ with $C$ large, where the summation runs over all zeroes $\rho$ of $L(s,f)$ that lie in the ball $B(1,1/\log q)$, listed according to their multiplicity.\footnote{The uniformity in $q$ here is of Linnik-type. For Dirichlet $L$-functions, see \cite[Proposition 18.5]{IK}.} Using partial summation, we find that
\[
\sum_{p\in I} \frac{f(p)}{p} \approx - \sum_{\rho} \int_I \frac{\dee y}{y^{2-\rho}\log y} .
\]
for any compact interval $I\subset[q^C,\infty)$. A calculation reveals that
\[
\int_I \frac{\dee y}{y^{2-\rho}\log y} 
	= \begin{cases}
			\int_I \frac{\dee y}{y\log y} +O(1)&\text{if}\ I\subset[e,e^{1/|\rho-1|}],\\
			O(1) &\text{if}\ I\subset[e^{1/|\rho-1|},\infty).
		\end{cases}
\]
We thus see that, when we examine the logarithmic sums $\sum_{p\in I}f(p)/p$, there is a sharp phase transition with regard to the influence of the zero $\rho$. This transition occurs at primes of size $e^{1/|\rho-1|}$. Hence, if $\rho_1,\dots,\rho_k$ are the zeroes of $L(s,f)$ in $B(1,1/\log q)$ listed according to their multiplicity and ordered so that
\[
|\rho_k-1|\ge \cdots \ge|\rho_1-1|,
\]
then we expect that 
\[
\sum_{p\in I} \frac{f(p)}{p} = - j \sum_{p\in I} \frac{1}{p}  + O(1) 
\]
for $j=0,1,\dots,k$ and for all compact intervals $I\subset[e^{1/|\rho_{j+1}-1|},e^{1/|\rho_j-1|})$, with the conventions that $\rho_{k+1}=1/\log q$ and $\rho_0=1$. 

To give a simple example, if $\chi$ is a primitive Dirichlet character of conductor $q>1$ whose $L$-function has a Siegel zero $\beta\in[1-1/\log q,1]$, then the above discussion suggests that
\[
\sum_{q<p\le e^{1/(1-\beta)}} \frac{1+\chi(p)}{p}=O(1)
\]
and
\[
\sum_{y<p\le z} \frac{\chi(p)}{p} = O(1) \quad( z\ge y \ge e^{1/(1-\beta)}) .
\]
This is indeed known. It follows, for example, from Proposition 8.5 in \cite{IK} or by combining Theorems 1.6 and 2.1 in \cite{dk-gafa}. 

In this note, we confirm the above heuristic in a much more general setting. 

\subsection{The class of functions $\CF_{A,D}(Q)$}

 Throughout, we let $f:\N\to\C$ denote a multiplicative function and we let $\Lambda_f$ be the arithmetic function defined via the convolution identity 
\[
f\log = f*\Lambda_f. 
\]
We shall assume that
\begin{equation}
	\label{eq:Lambda_f}
|\Lambda_f|\le D\cdot \Lambda
\end{equation}
for some fixed integer $D$. This implies that $|f|\le\tau_D$, as well as that $|g|\le \tau_D$ with $g$ denoting the Dirichlet inverse of $f$ (cf.~\cite[Lemma 2.2]{KS}). We could work with a weaker condition than \eqref{eq:Lambda_f}, but we choose this one for the sake of its striking simplicity. When $L(s,f)$ is an $L$-function in Selberg's class, \eqref{eq:Lambda_f} follows by Ramanujan's conjecture. 

We will further assume that there exists a parameter $Q\ge3$ and a constant $A>D+1$ such that 
\begin{equation}
\label{eq:f-small}
	\bigg| \sum_{n\le x} f(n)\bigg| \le x\cdot \frac{(\log Q)^{A-D-1}}{(\log x)^A} \quad(x\ge Q) .
\end{equation}
We think of $A$ as fixed. On the other hand, we allow $Q$ to vary; it plays an analogous role to that of the analytic conductor in the theory of $L$-function. Our results will be uniform in $Q$. In relation to the discussion of Section \ref{sec:heuristic}, condition \eqref{eq:f-small} is a weak analogue of $L(s,f)$ being an entire $L$-function of conductor $q$. 

We denote by $\CF_{A,D}(Q)$ the class of multiplicative functions $f$ satisfying \eqref{eq:Lambda_f} and \eqref{eq:f-small}. For such $f$, the Dirichlet series $L(s,f)$ converges for $\Re(s)\ge1$ in virtue of \eqref{eq:f-small}. Moreover, by work of the author and Soundararajan \cite[Proposition 2.4(c)]{KS}, we know that $L(s,f)$ has a zero of multiplicity $\le D$ at $s=1$. 

\subsection{New results}

We are now ready to state the main result of this paper.

\begin{theorem}
	\label{thm:main theorem} Fix $D\in\N$ and $A>D+1$. Let $Q\ge1$ and let $f\in\CF_{A,D}(Q)$. Let $m\le D$ denote the multiplicity of the zero of $L(s,f)$ at $s=1$. Then, there exist parameters
			\[
			Q=Q_{D+1}\le Q_D\le \cdots \le Q_{m+1}< Q_m=\infty
			\]
			such that, for each $j\in\{m,m+1,\dots,D\}$, we have that
			\[
			\sum_{p\in I} \frac{\Re(f(p))+j}{p} = O_{A,D}(1) 
			\]
		 uniformly over all compact intervals $I\subseteq[Q_{j+1},Q_j)$.
\end{theorem}

\begin{rem*}
The condition that $A>D+1$ is optimal. Indeed, if we consider the multiplicative function $f$ defined by the formula $\Lambda_f(p^k)=\big(-D+1_{p>100}/\log\log p\big)\Lambda(p^k)$ for all prime powers $p^k$, then it is possible to show that $|\sum_{n\le x}f(n)|\asymp_D x/(\log x)^{D+1}$ as $x\to\infty$. Moreover, $f$ does not satisfy the conclusion of Theorem \ref{thm:main theorem}. 
\end{rem*}

Unlike the situation in Section \ref{sec:heuristic}, there is no simple definition for the transition points $Q_j$ in Theorem \ref{thm:main theorem}; the proof does give explicit formulas for them, but they are not very enlightening. Note that our assumption \eqref{eq:f-small} does not guarantee the analytic continuation of $L(s,f)$ to the right of $\Re(s)=1$, so we cannot speak about zeroes of $L(s,f)$ in that region and relate their location to the size of the $Q_j$'s. Nonetheless, these transition points do exist, thus demonstrating that the heuristic argument described in Section \ref{sec:heuristic} is a very general phenomenon. In this sense, Theorem \ref{thm:main theorem} is very much in the spirit of the theory of {\it pretentious multiplicative functions}. 

On the other hand, if we assume a stronger version of \eqref{eq:f-small}, then it is possible relate the $Q_j$'s to the location of the zeroes of $L(s,f)$. To this end, we let $\CF^{\mathrm{strong}}_D(Q)$ denote the class of multiplicative functions $f$ such that $|\Lambda_f|\le D\Lambda$ and  
\begin{equation}
	\label{eq:f-small strong}
	 \bigg| \sum_{n\le x} f(n)\bigg| \le   \frac{x^{1-1/\log Q}}{(\log x)^{D+1}} \quad (x\ge Q) .
\end{equation}
For these functions $f$, the Dirichlet series $L(s,f)$ converges in the half-plane $\Re(s)\ge1-1/\log Q$ and the parameters $Q_j$ of Theorem \ref{thm:main theorem} can be related to the location of the zeroes of $L(s,f)$ in the ball $B(1,1/\log Q)$. 

\begin{theorem}
	\label{thm:zeroes}
	Let $f\in\CF_D^{\mathrm{strong}}(Q)$. 	There exists a constant $c_0=c_0(D)\in(0,1]$ such that $L(s,f)$ has at most $D$ zeroes in the ball $B(1,c_0/\log Q)$ counted according to their multiplicity. Moreover, if we let $m$ denote the multiplicity of the zero of $L(s,f)$ at $s=1$, and if we let $\rho_{m+1},\dots,\rho_d$ be the remaining zeroes of $L(s,f)$ in $B(1,c_0/\log Q)$ listed so that
	\[
	0<|\rho_{m+1}-1|\le |\rho_{m+2}-1|\le \cdots \le |\rho_d-1|\le \frac{c_0}{\log Q},
	\]
	then, for each $j\in\{m,m+1,\dots,d\}$, we have that
	\[
	\sum_{p\in I} \frac{\Re(f(p))+j}{p} = O_{A,D}(1) 
	\]
	uniformly over all compact intervals $I\subseteq[e^{1/|\rho_{j+1}|}, e^{1/|\rho_j-1|})$, with the conventions that $\rho_m=1$ and $\rho_{d+1}=1+c_0/\log Q$. 	
\end{theorem}

This theorem generalizes and makes rigorous the heuristic described in Section \ref{sec:heuristic}. A useful aspect of is that we may combine it with the theory of pretentious multiplicative functions to prove zero-free regions for various $L$-functions that are of similar strength to the classical ones. This is explained in Chapter 22 of \cite{dk-book} for the case of Dirichlet $L$-functions. For higher degree $L$-functions, we need to know Ramanujan's conjecture to verify \eqref{eq:Lambda_f}. Given that Ramanujan's conjecture remains wide open, it would be interesting to study to what extent we can relax \eqref{eq:Lambda_f} and still obtain Theorem \ref{thm:zeroes}. This question goes beyond the scope of the present paper; we will return to it in a future paper.


\subsection{Relation to earlier work}

The case $D=1$ of Theorem \ref{thm:main theorem} follows by \cite[Theorem 2.1]{dk-gafa}. That result is a bit stronger when $j=D=1$, because it implies that $\sum_{Q_2\le p< Q_1}|1+f(p)|/p\ll_A1$ in that case. Such an improvement is possible in the general case $j=D\ge1$, but we do not pursue it in this paper. The proof in \cite{dk-gafa} is quite involved. A much simpler proof of the special case when $f(n)=\chi(n)n^{it}$ with $\chi$ a Dirichlet character and $t\in\R$ is given in \cite[Theorem 22.5]{dk-book}. It is that proof which forms the basis of the argument we use to establish Theorem \ref{thm:main theorem}. 

A non-uniform version of Theorem \ref{thm:main theorem} follows by Lemma 4.1 in \cite{KS}. According to that result, we have
\begin{equation}
	\label{eq:KS}
	\sum_{p\le x} \frac{\Re(f(p))+m}{p} =O_f(1)
\end{equation}
for all $x\ge1$, where $m$ is as in the statement of Theorem \ref{thm:main theorem}. The main difference between the above estimate and Theorem \ref{thm:main theorem} is that the implicit constants in the latter depend at most on $A$ and $D$, and not on the specific values of $f$ as in \eqref{eq:KS}. This is an important distinction when using Theorem \ref{thm:main theorem}  in conjunction with Theorem \ref{thm:zeroes} to establish zero-free regions for $L$-functions.

It is also important to mention a related result of Sachpazis \cite{Sachpazis}. Building on earlier work of the author and Soundararajan \cite{dk-gafa,KS}, he proved that if $f\in\CF_{A,D}(Q)$ with $A>D+1$, then there exist real numbers $\gamma_1,\dots,\gamma_d$ with $d\le D$ such that 
\begin{equation}
	\label{eq:Sachpazis}
\lim_{x\to\infty} \frac{1}{x} \sum_{p\le x} \big( f(p)+p^{i\gamma_1}+\cdots+p^{i\gamma_d}\big) \log p = 0 . 
\end{equation}
In addition, the numbers $1+i\gamma_j$ are exactly the zeroes of $L(s,f)$ on the line $\Re(s)=1$ listed according to their multiplicity. As with \eqref{eq:KS}, Sachpazis's result depends on $f$. The asymptotic formula he proves is valid for $x$ large enough in terms of $f$. This shows a close connection zeroes of $L(s,f)$ close to the line $\Re(s)=1$. 

The zeroes $1+i\gamma_j$ partly explain Theorem \ref{thm:main theorem}. Indeed, if $\gamma\in[-C,C]$, then Lemma 9.1 in \cite{GKM} implies that
\[
\sum_{y<p\le z} \frac{1}{p^{1+i\gamma}} 
	= 
		\begin{cases}
				\log(\log z/\log y) + O_C(1) &\text{if}\ y\le z\le e^{1/|\gamma|},\\
				O_C(1) &\text{if}\ z\ge y\ge e^{1/|\gamma|}.
		\end{cases}
\]
However, as Theorem \ref{thm:zeroes} shows, the zeroes of $L(s,f)$ on the line $\Re(s)=1$ are not enough to prove Theorem \ref{thm:main theorem}.

Lastly, when $D=1$, it is worth noticing that we know a more uniform version of \eqref{eq:Sachpazis} that is more in line with Theorem \ref{thm:main theorem} - see Theorem 1.1. in \cite{dk-gafa}. We expect that such an estimate can be proven for general $D$ by combining the methods of this paper and \cite{dk-gafa}.

\section{Sieve estimates}

\begin{lemma}\label{lem:sieving tau}
Let $k\in\N$, let $y\ge3/2$, and let $V(y)=\prod_{p\le y}(1-1/p)$. There are real numbers $c_{k,j}(y)=O_k(1)$ for $j=0,1,\dots,k-1$ such that
\[
\sum_{\substack{n\le x \\ P^-(n)>y}} \tau_k(n)
	 = xV(y)
	 	\sum_{j=0}^{k-1} c_{k,j}(y) \bigg(\frac{\log x}{\log y}\bigg)^{k-1-j}
			+ O_k\bigg( \frac{x^{1-2^{-k}/\log y}}{\log y} \bigg) 
\]
uniformly for all $x\ge y$. 
\end{lemma}


\begin{proof} This follows by a routine application of Dirichlet's hyperbola method. The technical details are a bit complicated, so we give a full proof. 
	
We use induction on $k$ to prove the lemma. When $k=1$, the Fundamental Lemma of Sieve Methods (see, for example, Theorem 19.1 in \cite{dk-book}) implies that
	\begin{equation}
		\label{eq:sieving tau k=1}
	\sum_{\substack{n\le x \\ P^-(n)>y}} \tau_k(n)
	= xV(y) + E_y(x),
	\end{equation}
	where $E_y(x)\ll_c x^{1-1/\log y}/\log y$ for all $x\ge y\ge3$. So the lemma follows with $c_{0,0}(y)=1$. 
	
Next assume that the lemma holds with $k-1$ in place of $k$, and let us prove it for $k$. 	
	Note that $x(\log x)^j -y (\log y)^j  = \int_y^x (\log w)^j \dee w +  j\int_y^x (\log w)^{j-1}\dee w$ for $j\in\N$.  Iterating this identity and combining it with our induction hypothesis, we find that there exist real numbers $d_{k-1,j}(y)=O_k(1)$ for $j=0,1,\dots,k-2$ such that
\begin{equation}
	\label{eq:tau induction hyp}
	\sum_{\substack{n\le x \\ P^-(n)>y}} \tau_{k-1}(n)
	= V(y)\int_y^x 
	\sum_{j=0}^{k-2} d_{k-1,j}(y)  \bigg(\frac{\log w}{\log y}\bigg)^{k-2-j} \dee w
	+ R_{k-1,y}(x),
\end{equation}
where $R_{k-1,y}(x)\ll_k x^{1-2^{-k+1}/\log y}  / \log y$ for all $x\ge y$. Let us now see how to use this formula to conclude the proof.

First of all, when $x\le y^2$, then conditions $n\le x$ and $P^-(n)>y$ imply that $n\in\{1\}\cup\{y<p\le x\}$. Hence, the lemma follows in this case by Chebyshev's estimate. 

Let us assume now that $x\ge y^2$. We have
\begin{align*}
\sum_{\substack{n\le x \\ P^-(n)>y}} \tau_k(n)
	&=\sum_{\substack{b\le \sqrt{x} \\ P^-(b)>y}} 
		\sum_{\substack{\sqrt{x}<a\le x/b \\  P^-(a)>y}} \tau_{k-1}(a) 
	+ \sum_{\substack{a\le \sqrt{x}\\ P^-(a)>y}} \tau_{k-1}(a) \sum_{\substack{b\le x/a \\P^-(b)>y}} 1.
\end{align*}
Using \eqref{eq:sieving tau k=1} and \eqref{eq:tau induction hyp}, we find that
\begin{equation}\label{eq:tau hyperbola method}	
	\begin{split}
\sum_{\substack{n\le x \\ P^-(n)>y}} \tau_k(n)
	&= V(y)  \sum_{\substack{b\le \sqrt{x} \\ P^-(b)>y}} 
		\int_{\sqrt{x}}^{x/b} 
	 	\sum_{j=0}^{k-2} d_{k-1,j}(y) \left(\frac{\log w}{\log y}\right)^{k-2-j}  \dee w \\
	&\qquad +xV(y) 
		\sum_{\substack{a\le \sqrt{x}\\ P^-(a)>y}} \frac{\tau_{k-1}(a)}{a} 
	+  O_k\left(\frac{x^{1-2^{-k}/\log y}}{\log y}\right)  .
	\end{split}
\end{equation}

We first estimate the sum over $a$. By \eqref{eq:tau induction hyp} and partial summation, we have
\[
\begin{split}
\sum_{\substack{a\le \sqrt{x}\\ P^-(a)>y}} \frac{\tau_{k-1}(a)}{a} 
	&= 1+ \int_y^{\sqrt{x}} \frac{1}{w} \dee \sum_{\substack{a\le w\\ P^-(a)>y}} \tau_{k-1}(a)  \\
	&= 1+ V(y)\int_y^{\sqrt{x}} \sum_{j=0}^{k-2} d_{k-2,j}(y)  \bigg(\frac{\log w}{\log y}\bigg)^{k-2-j} \frac{\dee w}{w} 
	+   \int_y^x \frac{1}{w} \dee R_{k-1,y}(w) .
\end{split}
\]
Moreover,
\[
\begin{split}
 \int_y^x \frac{1}{w} \dee R_{k-1,y}(w) 
 	&= \frac{R_{k-1,y}(w)}{w}\bigg|_{w=y}^x  + \int_y^x \frac{R_{k-1,y}(w)}{w^2} \dee w  \\
 	&= \int_y^\infty \frac{R_{k-1,y}(w)}{w^2} \dee w  + O_k(x^{1-2^{-k}/\log y}) .
\end{split}	
\]
We conclude that
\begin{equation}\label{eq:tau sum over a}
\begin{split}
	\sum_{\substack{a\le \sqrt{x}\\ P^-(a)>y}} \frac{\tau_{k-1}(a)}{a} 
	&=V(y)\int_y^{\sqrt{x}} \sum_{j=0}^{k-2} d_{k-2,j}(y)  \bigg(\frac{\log w}{\log y}\bigg)^{k-2-j} \frac{\dee w}{w} \\
	&\quad + C_{k,y} + O_k(x^{1-2^{-k}/\log y}) ,
\end{split}
\end{equation}
where
\[
C_{k,y}:=1+\int_y^\infty \frac{R_{k-1,y}(w)}{w^2} \dee w \ll_k 1
\]

Finally, we estimate the sum over $b$. For any $\ell\in\Z_{\ge0}$, we have
\[
\sum_{\substack{b\le \sqrt{x} \\ P^-(b)>y}} \int_{\sqrt{x}}^{x/b} \bigg(\frac{\log w}{\log y}\bigg)^\ell  \dee w 
	= \int_{\sqrt{x}}^{x}\sum_{\substack{b\le x/w \\ P^-(b)>y}} \bigg(\frac{\log w}{\log y}\bigg)^\ell \dee w.
\]
If $w>x/y$, the conditions $b\le x/w$ and $P^-(b)>y$ imply that $b=1$. Otherwise, we have $x/w\ge y$. Using \eqref{eq:sieving tau k=1}, we find that
\begin{align*}
\sum_{\substack{b\le \sqrt{x} \\ P^-(b)>y}} \int_{\sqrt{x}}^{x/b} \bigg(\frac{\log w}{\log y}\bigg)^\ell  \dee w 
		&=
		 \int_{\sqrt{x}}^{x/y} \frac{V(y)x}{w}\bigg(\frac{\log w}{\log y}\bigg)^\ell  \dee w  
		+ \int_{\sqrt{x}}^{x/y} E_y(x/w) \bigg(\frac{\log w}{\log y}\bigg)^\ell\dee w \\
		&\qquad +  \int_{x/y}^x \bigg(\frac{\log w}{\log y}\bigg)^\ell \dee w  .
\end{align*}
We have
\begin{align*}
	\int_{\sqrt{x}}^{x/y} E_y(x/w) \bigg(\frac{\log w}{\log y}\bigg)^\ell\dee w
		&= x\int_{y}^{\sqrt{x}} E_y(u) \bigg(\frac{\log(x/u)}{\log y}\bigg)^\ell \frac{\dee u}{u^2}\\
		&= x\sum_{i=0}^\ell \binom{\ell}{i} \frac{(-1)^i (\log x)^{\ell-i}}{(\log y)^\ell} \int_{y}^{\sqrt{x}} \frac{E_y(u) (\log u)^i}{u^2} \dee u .
\end{align*}
Moreover, 
\[
\int_{\sqrt{x}}^\infty \frac{E_y(u) (\log u)^i}{u^2} \dee u \ll_j (\log x)^j x^{-2^{-1}/\log y},
\]
as well as
\[
e_i(y) := \frac{1}{(\log y)^i} \int_y^\infty \frac{E_y(u) (\log u)^i}{u^2} \dee u  \ll_i 1. 
\]
Therefore 
\[
	\int_{\sqrt{x}}^{x/y} E_y(x/w) \bigg(\frac{\log w}{\log y}\bigg)^\ell\dee w
	= x\sum_{i=0}^\ell \binom{\ell}{i} (-1)^i e_i(y) \bigg(\frac{\log x}{\log y}\bigg)^{\ell-i} + O_\ell\big( x^{-2^{-k}/\log y}\big),
\]
where we used that $k\ge2$. Lastly, 
\[
\int_{x/y}^x \bigg(\frac{\log w}{\log y}\bigg)^\ell \dee w
	= x\int_1^y \bigg(\frac{\log(x/u)}{\log y}\bigg)^\ell \frac{\dee u}{u^2}
	= x\sum_{i=0}^\ell \binom{\ell}{i} (-1)^i f_i(y)\bigg(\frac{\log x}{\log y}\bigg)^{\ell-i},
\]
where
\[
f_i(y):= \frac{1}{(\log y)^i} \int_1^y \frac{(\log u)^i}{u^2} \dee u \le \frac{i!}{(\log y)^i}\ll_i 1. 
\]
We conclude that
\begin{equation}\label{eq:tau sum over b}	
	\begin{split}
			\sum_{\substack{b\le \sqrt{x} \\ P^-(b)>y}} \int_{\sqrt{x}}^{x/b} \bigg(\frac{\log w}{\log y}\bigg)^\ell  \dee w 
		&=
		\int_{\sqrt{x}}^{x/y} \frac{V(y)x}{w}\bigg(\frac{\log w}{\log y}\bigg)^\ell  \dee w   \\
		&\quad +  x\sum_{i=0}^\ell \binom{\ell}{i} (-1)^i \big(e_i(y)+f_i(y)\big) \bigg(\frac{\log x}{\log y}\bigg)^{\ell-i} .
	\end{split}
\end{equation}

Combining \eqref{eq:tau hyperbola method}, \eqref{eq:tau sum over a} and \eqref{eq:tau sum over b}, we find that
\begin{align*}
	\sum_{\substack{n\le x \\ P^-(n)>y}} \tau_k(n)
	&= xV(y)^2
	\int_{y}^{x/y} 
	\sum_{j=0}^{k-2} d_{k-1,j}(y) \left(\frac{\log w}{\log y}\right)^{k-2-j} \frac{\dee w}{w} 
	+xV(y) C_{k,y} \\
	&\qquad + xV(y) \sum_{j=0}^{k-2} d_{k-1,j}(y) \sum_{i=0}^{k-2-j}(-1)^i\big(e_i(y)+f_i(y)\big) \bigg(\frac{\log x}{\log y}\bigg)^{\ell-i} \\
	&\qquad +  O_k\left(\frac{x^{1-2^{-k}/\log y}}{\log y}\right)  
\end{align*}
for all $x\ge y^2$. Since 
\[
\begin{split}
V(y)\int_y^{x/y}\bigg(\frac{\log w}{\log y}\bigg)^\ell\frac{\dee w}{w} 
&= \frac{V(y)(\log(x/y))^{\ell+1}}{(\log y)^\ell} - \frac{V(y)\log y}{\ell+1} \\
	&= \sum_{i=0}^{\ell+1} \binom{\ell+1}{i} \frac{(-1)^i V(y)\log y}{\ell+1} \bigg(\frac{\log x}{\log y}\bigg)^{\ell-i} 
	- \frac{V(y)\log y}{\ell+1} ,
\end{split}
\]
the inductive step is complete. 
\end{proof}

\begin{lemma}
	\label{lem:L(s,f)}
	Let $D\in\N$ and let $f$ be a multiplicative function such that $|\Lambda_f|\le D\Lambda$. For $z\ge y\ge 3$, we have
	\[
	\log|L_y(1+1/\log z,f)| = \sum_{y<p\le z}\frac{\Re(f(p))}{p} + O(1)
	\]
\end{lemma}

\begin{proof}
This lemma is due to Granville and Soundararajan. We sketch the proof for completeness. We write 
\[
\log|L_y(1+\tfrac{1}{\log z},f)| =\sum_{p>y} \sum_{k\ge1} \frac{\Lambda_f(p^k)}{p^{k(1+\frac{1}{\log z})}\log(p^k)}. 
\]
The terms with $k\ge2$ contribute $O_D(1)$. The terms with $k=1$ and $p>z$ contribute also $O_D(1)$, because $\sum_{p>z}1/p^{1+1/\log z}=O(1)$. Finally, we write $p^{-1/\log z}=1+O(\log p/\log z)$ when $p\le z$ to complete the proof.
\end{proof}

\begin{lemma}\label{lem:sifted-f}
Let $D\in\N$, $A\ge D+1$, $\eta\in[0,1]$ and $Q\ge3$. Let $f:\N\to\C$ be a multiplicative function such that  $|\Lambda_f|\le D\Lambda$ and 
\[
\bigg| \sum_{n\le x} f(n) \bigg| \le x^{1-\eta/\log Q}\cdot \frac{(\log Q)^{A-D-1}}{(\log x)^A} . 
\]
There exists a constant $\alpha=\alpha(D)\in(0,1/2]$ such that
\[
\sum_{\substack{n\le x \\ P^-(n)>y}} f(n) \ll_{A,D} x^{1-\alpha\eta/\log y}\cdot \frac{(\log y)^{A-1}}{(\log x)^A} \quad\text{for}\ x\ge y \ge Q.
\]
\end{lemma}

\begin{proof} When $D=1$, the needed estimate follows by Lemma 4.2 in \cite{dk-gafa}. For general $D$, we modify an argument of Sachpazis \cite[Proposition 4.3]{Sachpazis}. We indicate only the necessary changes. 

Let $c=c(D)$ be the constant appearing in Lemma 4.1 of \cite{Sachpazis}, and let $C=\min\{c/2,1/16\}$.  When $x\le y^{(4D+1)/C}$, we simply note that 
\[
\bigg|\sum_{\substack{n\le x \\ P^-(n)>y}} f(n)\bigg|\le \bigg|\sum_{\substack{n\le x \\ P^-(n)>y}} \tau_d(n)\bigg| \ll \frac{x(\log x)^{D-1}}{(\log y)^D} \asymp_{D,A} \frac{x^{1-\eta/\log y}(\log y)^{A-1}}{(\log x)^A},
\]	
where we used Theorem 14.2 in \cite{dk-book} to get the second inequality. 

Let us now assume that $x\ge y^{(4D+1)/C}$. Let $k_D=\prod_{p\le D^3}p$. By Lemma 3.8 in \cite{Sachpazis}, and by the proof of Proposition 4.3 in the same paper, there exist functions $\lambda^\pm$ supported on $\{d\le x^C : d|\prod_{p\le y}p\}$ such that $|\lambda^\pm|\le 1$, $(1*\lambda^-)(n)\le 1_{P^-(n)>y}\le (1*\lambda^+)(n)$ and 
\[
\sum_{\substack{n\le x \\ (n,k_D)=1}} (\lambda^+*1-\lambda^-*1)(n)|f(n)|\ll_D x^{1-C/2}(\log x)^{D-1}+ \frac{x^{1-C/(2\log y)}}{\log y}. 
\]
(To get the above estimate, see the first displayed equation on page 2953 of \cite{Sachpazis}, which must be corrected by also including the contribution of the error term from the third displayed equation on page 2952.) Thus
\begin{equation}
	\label{eq:reduce to S}
\sum_{\substack{n\le x \\ P^-(n)>y}} f(n) = S+  O_{D,A}\bigg( \frac{x^{1-C/(3\log y)} (\log y)^{A-1}}{(\log x)^A} \bigg) ,
\end{equation}
where
\[
S= \sum_{\substack{n\le x \\ (n,k_D)=1}}(\lambda^+*1)(n) f(n)  .
\]
Let us now see how to bound $S$.

We have
\[
S= \sum_{\substack{d\le x^C \\ (d,k_D)=1}} \lambda^+(d) \sum_{\substack{d'\le x,\ d|d' \\ p|d'\ \iff\ p|d}} f(d') \sum_{\substack{m\le  x/d' \\ (m,dk_D)=1}} f(m) . 
\]
We may write $f(m)1_{(m,dk_D)=1}=(f*h_d)(m)$, where $h_d$ the multiplicative such that $h_d(p^k)=1_{p|dk_D} g(p^k)$ with $g$ denoting the Dirichlet inverse of $f$. We have $|g|\le \tau_D$ because $|\Lambda_f|\le D\Lambda$ \cite[Lemma 2.2]{KS}. Thus
\[
S=  \sum_{\substack{d\le x^C \\ (d,k_D)=1}} \lambda^+(d) \sum_{\substack{d'\le x,\ d|d' \\ p|d'\ \iff\ p|d}} f(d') 
\sum_a h_d(a) \sum_{r\le  x/(ad')} f(r) .
\]
Taking absolute values, we find that
\begin{equation}
	\label{eq:S bound}
|S|\le  \sum_{\substack{d\le x^C \\ d|\prod_{p\le y}p}}  \sum_{\substack{d'\le x,\ d|d' \\ p|d'\ \iff\ p|d}} \tau_D(d') 
\sum_{\substack{a\in\N \\ p|a\ \Rightarrow p|k_Dd}} \tau_D(a) \bigg| \sum_{r\le  x/(ad')} f(r) \bigg| = S_1 + S_2, 
\end{equation}
where $S_1$ denotes the subsum with $ad'\le \sqrt{x}$ and $S_2$ the subsum with $ad'>\sqrt{x}$. 

We first bound $S_1$. When $ad'\le \sqrt{x}$, we have $x/(ad')\ge \sqrt{x} \ge y^{\frac{4D+1}{2C}} \ge Q$, and thus 
\[
\sum_{r\le  x/(ad')} f(r) \ll_D \frac{(x/ad')^{1-\eta/\log Q} (\log Q)^{A-D-1}}{(\log x)^A} \le \frac{x^{1-\eta/(2\log Q)}(\log Q)^{A-D-1}}{ad'(\log x)^A} .
\]
We have 
\[
\sum_{a:\, p|a\,\Rightarrow\,p|k_Dd} \frac{\tau_D(a)}{a}\ll_D (d/\phi(d))^D
\quad\text{and}\quad 
\sum_{d':\, p|d'\, \Leftrightarrow\, p|d} \frac{\tau_D(d')}{d'} \le \prod_{p|d} \sum_{k\ge1} \frac{\tau_D(p^k)}{p^k}.
\]
Thus
\begin{equation}
	\label{eq:S_1}
	S_1\ll_D \frac{x^{1-\eta/(2\log Q)}(\log Q)^{A-D-1}(\log y)^D}{(\log x)^A}\le  \frac{x^{1-\eta/(2\log y)} (\log y)^{A-1}}{(\log x)^A}  
\end{equation}
for all $y\ge Q$. 

Next, we bound $S_2$. We have the trivial bound 
\[
\sum_{r\le  x/(ad')} f(r) \ll \frac{x(\log x)^{D-1}}{ad'} \le \frac{x^{5/6}(\log x)^{D-1}}{(ad')^{2/3}} . 
\] 
whenever $\sqrt{x}<ad'\le x$. Moreover,
\[
\sum_{a:\, p|a\,\Rightarrow\,p|k_Dd} \frac{\tau_D(a)}{a^{2/3}}\ll_D \prod_{p|d}(1-1/p^{2/3})^{-D} 
\quad\text{and}\quad 
\sum_{d':\, p|d'\, \Leftrightarrow\, p|d} \frac{\tau_D(d')}{(d')^{2/3}} \le \prod_{p|d} \sum_{k\ge1} \frac{\tau_D(p^k)}{p^{2k/3}}.
\]
We conclude that
\[
S_2\ll_D x^{5/6}(\log x)^{D-1} \sum_{d\le x^C} \frac{\prod_{p|d}(1-1/p^{2/3})^{-D}}{d^{2/3}} .
\]
Thus
\begin{equation}
	\label{eq:S_2}
S_2\ll_D x^{C/3+5/6}(\log x)^{D-1} \sum_{d\le x^C} \frac{\prod_{p|d}(1-1/p^{2/3})^{-D}}{d} \ll x^{C/3+5/6}(\log x)^{2D-1} . 
\end{equation}
Since $|S|\le S_1+S_2$ by \eqref{eq:S bound}, relations \eqref{eq:reduce to S}, \eqref{eq:S_1} and \eqref{eq:S_2} imply that
\[
\sum_{\substack{n\le x \\ P^-(n)>y}} f(n) 
	\ll _D\frac{x^{1-\eta/(2\log y)}  (\log y)^{A-1}}{(\log x)^A} + x^{C/3+5/6}(\log x)^{2D-1} +  \frac{x^{1-C/(3\log y)} (\log y)^{A-1}}{(\log x)^A}. 
\]
This completes the proof of the lemma upon taking $\alpha=C/3$.
\end{proof}

\begin{lemma}\label{lem:L^{(j)}(1,f)}
Fix $D\in\N$ and $A>D+1$. Let $Q\ge3$ and let $f\in\CF_{A,D}(Q)$. 
\begin{enumerate}
	\item For $\sigma\ge1$, $y\ge Q$ and $j=0,1,\dots,D$, we have $L^{(j)}(\sigma,f)\ll_{A,D} (\log y)^j$. 
	\item For $z\ge y\ge Q$ and  $j=0,1,\dots,D$, we have 
\[
L_y^{(j)}(1,f) \ll_{A,D} (\log z)^j |L_y(1+1/\log z,f)| .
\]
\end{enumerate}
\end{lemma}

\begin{proof}(a) This follows readily by Lemma \ref{lem:sifted-f} and partial summation.
	
	\medskip

	(b) We write $L_y(s,f)=F(s)L_z(s,f)$. First of all, we know that $L_z^{(i)}(\sigma,f)\ll_{A,D} (\log z)^i$ for $i=0,1,\dots,j$ by part (a). Moreover, we have $F=e^G$ with $G=\log F$. Hence, Fa\`a di Bruno's formula implies the identity
	\[
	\frac{F^{(\ell)}(s)}{F(s)} = j! \sum_{a_1+2a_2+\cdots+\ell a_\ell=\ell} \prod_{i=1}^\ell \bigg\{\frac{1}{a_i!}	\left(\frac{-F'}{i!F}(s)\right)^{a_i} \bigg\}.
	\]
	Since
	\[
	\left(\frac{-F'}{F}\right)^{(i-1)}(\sigma) = \sum_{p|n\ \Rightarrow\ y<p\le z} \frac{\Lambda_f(n)(-\log n)^{i-1}}{n^\sigma}
	\ll (\log z)^i
	\]
	for all $i\ge1$	by our assumption that $|\Lambda_f|\le D\cdot \Lambda$, we conclude that $(F^{(\ell)}/F)(\sigma)\ll_\ell(\log z)^\ell$ for all $\ell\ge1$. Putting together the above estimates implies that
	\[
	L_y^{(j)}(1,f) \ll_{A,D} (\log z)^j |F(1+1/\log z,f)| .
	\]
	Lastly, arguing as in Lemma \ref{lem:L(s,f)}, we find that
	\[
	|F(1+1/\log z)|\asymp_D |L_y(1+1/\log z,f)|.
	\]
	This completes the proof.
\end{proof}

\begin{lemma}\label{lem:sieving tau_k*f}
	Fix $D\in\N$, $k\in\{0,1,\dots,D\}$ and $A>D+1$. Let $Q\ge3$ and $f\in\CF_{A,D}(Q)$. For each $y\ge Q$, there exists real numbers $c_{k,0}(y,f),c_{k,1}(y,f),\dots,c_{k,k-1}(y,f)$ such that 
	\begin{equation}
		\label{eq:c_{k,j}(y,f)}
		c_{k,j}(y,f) \ll_{A,D} \frac{1}{(\log y)^{k-j}} \sum_{i=0}^j \frac{|L_y^{(i)}(1,f)|}{(\log y)^i}
	\end{equation}
	and 
	\[
	\sum_{\substack{n\le x \\ P^-(n)>y}} (\tau_k*f)(n)
	= \int_0^x \sum_{j=0}^{k-1} c_{k,j}(y,f)(\log w)^{k-1-j} \dee w
	+ O_{A,D}\bigg( \frac{x(\log y)^{A-k-1}}{(\log x)^{A-k}} \bigg)
	\]
	uniformly for all $x\ge y$. In particular, if $L(s,f)$ has a zero of multiplicity $m$ at $s=1$, then $c_{k,j}(y,f)=0$ for all $j<m$. 
\end{lemma}


\begin{proof} The result follows by a routine application of Dirichlet's hyperbola method, using Lemmas \ref{lem:sifted-f} and \ref{lem:sieving tau}. We give the details for completeness.
	
	Note that $x(\log x)^{k-j-1} = \int_0^x (\log w)^{k-j-1} \dee w +1_{j\le k-2}\cdot (k-j-1) \int_0^x (\log w)^{k-j-2}\dee w$ for $j=0,1,\dots,k-1$. Moreover, the right-hand side of \eqref{eq:c_{k,j}(y,f)} is $\ll_{A,D} 1/(\log y)^{k-j}$ by Lemma \ref{lem:L^{(j)}(1,f)}(a). Thus it suffices to show that 
	there are real numbers $d_{k,0}(y,f),d_{k,1}(y,f),\dots,d_{k,k-1}(y,f)$ such that 
	\begin{equation}
		\label{eq:d_{k,j}(y,f) bound}
		d_{k,j}(y,f) \ll_{A,D} \frac{1}{(\log y)^{k-i}} \sum_{i=0}^j \frac{|L_y^{(i)}(1,f)|}{(\log y)^i}
	\end{equation}
	and 
	\begin{equation}
		\label{eq:tau_k*f goal}
		\sum_{\substack{n\le x \\ P^-(n)>y}} (\tau_k*f)(n)
		= x\sum_{j=0}^{k-1} d_{k,j}(y,f)(\log x)^{k-1-j} 
		+ O_{A,D}\bigg( \frac{x(\log y)^{A-k-1}}{(\log x)^{A-k}} \bigg)
	\end{equation}
	uniformly for all $x\ge y$. Indeed, the lemma would then follow with
	\[
	c_{k,j}(y,f) = d_{k,j}(f) + 1_{j\ge1} \cdot (k-j) d_{k,j-1}   \ll_{A,D} \frac{1}{(\log y)^{k-j}} \sum_{i=0}^j \frac{|L_y^{(i)}(1,f)|}{(\log y)^i} .
	\]
	
	Let us now prove \eqref{eq:tau_k*f goal}. Recall Lemma \ref{lem:sieving tau} and the real numbers $c_{k,j}(y)\ll_k1$ in its statement. We then set
		\begin{equation}
		\label{eq:tau_k coeff bound}
		\tilde{c}_{k,\ell}(y) = \frac{c_{k,j}(y)\prod_{p\le y}(1-1/p)}{(\log y)^{k-\ell-1}} \ll \frac{1}{(\log y)^{k-\ell}} 
	\end{equation}
	and
		\begin{equation}
		\label{eq:d_{k,j}(y,f) definition}
	d_{k,j}(y,f) = \sum_{\ell=0}^j \binom{k-j+i-1}{i}  \tilde{c}_{k,j-i}(y) L^{(i)}(1,f),
	\end{equation}
	which satisfies \eqref{eq:d_{k,j}(y,f) bound} by \eqref{eq:tau_k coeff bound}. We will prove \eqref{eq:tau_k*f goal} holds with this choice of $d_{k,j}(y,f)$. 
	
	First of all, when $x\le y^2$, then conditions $n\le x$ and $P^-(n)>y$ imply that $n\in\{1\}\cup\{y<p\le x\}$. Hence, \eqref{eq:tau_k*f goal} follows in this case by Chebyshev's estimate. 
	
	Let us now assume that $x\ge y^2$. We have
	\begin{align*}
		\sum_{\substack{n\le x \\ P^-(n)>y}} (\tau_k*f)(n)
		&=\sum_{\substack{a\le \sqrt{x} \\ P^-(b)>y}} f(a) \sum_{\substack{b\le x/a \\ P^-(b)>y}} \tau_k(b) 
		+ \sum_{\substack{b\le \sqrt{x}\\ P^-(b)>y}} \tau_k(b) \sum_{\substack{\sqrt{x}<a\le x/b \\P^-(a)>y}} f(a).
	\end{align*}
	The rightmost double sum is $\ll x(\log y)^{A-D-1}/(\log x)^{A-D}$ by Lemma \ref{lem:sifted-f}. To estimate the first double sum on the right-hand side, we use Lemma \ref{lem:sieving tau}. Recalling the definition of $\tilde{c}_{k,\ell}(y)$ from \eqref{eq:tau_k coeff bound}, we find that
	\begin{equation}\label{eq:tau_k*f hyperbola method}	
		\begin{split}
			\sum_{\substack{n\le x \\ P^-(n)>y}} (\tau_k*f)(n)
			&= x\sum_{\ell=0}^{k-1} \tilde{c}_{k,\ell}(y) \sum_{\substack{a\le \sqrt{x} \\ P^-(a)>y}}  \frac{f(a)(\log(x/a))^{k-\ell-1}}{a} \\
			&\qquad +O_{A,D}\bigg(\frac{x(\log y)^{A-k-1}}{(\log x)^{A-k}}\bigg) .
		\end{split}
	\end{equation}
	By Lemma \ref{lem:sifted-f} and partial summation, we have
	\[
	\sum_{\substack{a> \sqrt{x} \\ P^-(a)>y}}  \frac{f(a)(\log(x/a))^{k-\ell-1}}{a} \ll_{A,D} (\log x)^{k-\ell-1} \cdot \bigg(\frac{\log y}{\log x}\bigg)^{A-1} ,
	\]
	where we used that $A>D+1\ge (k-\ell-1)+2$. Together with \eqref{eq:tau_k*f hyperbola method} and \eqref{eq:tau_k coeff bound}, this implies 
	\[
	\begin{split}
		\sum_{\substack{n\le x \\ P^-(n)>y}} (\tau_k*f)(n)
		&= x\sum_{\ell=0}^{k-1} \tilde{c}_{k,\ell}(y) \sum_{P^-(a)>y}  \frac{f(a)(\log(x/a))^{k-\ell-1}}{a} +O_{A,D}\bigg(\frac{x(\log y)^{A-k-1}}{(\log x)^{A-k}}\bigg) .
	\end{split}
	\]
	Finally, using the binomial theorem, we find that 
	\[
	(\log(x/a))^{k-\ell-1} = \sum_{\substack{0\le i \le j\le k-1 \\  j=i+\ell}} \binom{k-\ell-1}{i} (\log x)^{k-j-1} (-\log a)^i .
	\]
	Using \eqref{eq:d_{k,j}(y,f) definition}, we find that
	\begin{align*}
		\sum_{\ell=0}^{k-1} \tilde{c}_{k,\ell}(y) \sum_{P^-(a)>y}  \frac{f(a)(\log(x/a))^{k-\ell-1}}{a}  = \sum_{j=0}^{k-1} d_{k,j}(y,f)(\log x)^{k-j-1} .
	\end{align*}
This completes the proof of \eqref{eq:tau_k*f goal}, and thus of the lemma.
\end{proof}

\begin{lemma}\label{lem:tau_m*f}
Fix $D\in\N$ and $A>D+1$. Let $y\ge Q\ge3$ and $f\in\CF_{A,D}(Q)$, and assume that $L(s,f)$ has a zero of multiplicity $m\le D$ at $s=1$. 
\begin{enumerate}
\item For $\sigma\ge1$ and $j\in\{0,1,\dots,D-m\}$, we have
\[
L_y^{(j)}(\sigma,\tau_m*f) \ll_{A,D} (\log y)^j  .
\]
\item For $\sigma\ge1$, we have $L_y(1,\tau_m*f)\ll |L_y(\sigma,\tau_m*f)|$. 
\item Assume that $m\le D-1$, and let $k\in\{m,\dots,D-1\}$. Uniformly for $z\ge y\ge Q$ and $\sigma\in[1+1/\log z,1+1/\log y]$, we have
\[
L_y'(\sigma,\tau_k*f) \ll_{A,D}\frac{|L_y(1+1/\log z,\tau_k*f)|}{(\sigma-1)^2\log z} + \log y . 
\]
\end{enumerate}
\end{lemma}


\begin{proof} (a) Using Lemma \ref{lem:sieving tau_k*f}, we have
	\[
	\sum_{\substack{n\le x \\ P^-(n)>y}} (\tau_m*f)(n) \ll_{A,D} \frac{x(\log y)^{A-m-1}}{(\log x)^{A-m}},
	\]
	because $c_{m,j}(y)=0$ for $j=0,1,\dots,m-1$. Hence, the claimed estimate on $L_y^{(j)}(\sigma,\tau_m*f)$ follows by partial summation.
	
	\medskip
	
	(b) Let $g=\tau_m*f$. If $\sigma\ge 1+1/\log y$, then we have $|L_y(\sigma,g)|\asymp_D 1$. Since $L_y(1,f)\ll1$ by part (a), the needed result follows. Finally, assume that $\sigma=1+1/\log z$ with $z\ge y$. For every $x\ge z$, applying Lemma \ref{lem:L(s,f)} three times with $g$ implies that
	\begin{align*}
	\log|L_y(1+1/\log x,g)|  
		&= \log |L_y(1+1/\log z,g)|+ \log |L_z(1+1/\log x,g)| + O(1)  \\
		&\le \log|L_y(1+1/\log z,g)|+O(1),
		\end{align*}
	where we used part (a). Letting $x\to\infty$ and exponentiating completes the proof of part (b). 
	
	\medskip
	
(c) Let $\sigma=1+1/\log w$, so that $w\in[y,z]$. Using Lemma \ref{lem:sieving tau_k*f} and partial summation, we have
	\[
	-L_y(\sigma,\tau_k*f) = \sum_{P^-(n)>y} \frac{(\tau_k*f)(n)\log n}{n^{1+1/\log w}} 
	= \int_y^\infty \sum_{j=0}^{k-1} c_{k,j}(y,f) \frac{(\log x)^{k-j}}{x^{1+1/\log w}} \dee x 
	+ O_{A,D}(\log y) . 
	\]
	Letting $x=w^t$, we find that
	$
	\int_y^\infty (\log x)^{k-j}/x^{1+1/\log w} \dee x\ll_D (\log w)^{k-j+1}$, and thus 
	\[
	L_y'(\sigma,\tau_k*f) \ll_{A,D}  \sum_{j=0}^{k-1} |c_{k,j}(y,f)| (\log w)^{k-j+1} + \log y .
	\]
	We use the bound $c_{k,j}(y,f)$ from Lemma \ref{lem:sieving tau_k*f} in conjunction with Lemma \ref{lem:L^{(j)}(1,f)} to find that
	\[
	c_{k,j}(y,f) \ll_{A,D} \sum_{i=0}^j \frac{|L_y^{(i)}(1,f)|}{(\log y)^{k-j+i}} 
		\ll_D \frac{|L_y(1+1/\log z,f)|(\log z)^j}{(\log y)^k} . 
	\]
	Putting together the above estimates implies that
	\[
	L_y'(\sigma,\tau_k*f) 
		\ll \frac{|L_y(1+1/\log z,f)|(\log z)^{k-1} (\log w)^2}{(\log y)^k} + \log y .
	\]
	To complete the proof, note that $L_y(1+1/\log z,\tau_k) \asymp(\log z/\log y)^k$.
\end{proof}



\section{Proof of Theorem \ref{thm:main theorem}}

We are finally ready to prove our main result. All implied constants might depend on $A$ and $D$. Throughout, we set 
\[
g=\tau_m*f.
\]
Moreover, we will write $Q_{m+j}=Y_j$ for $j=1,\dots,D-m$, where the $Y_j$s are to be defined. We also adopt the convention that $Y_0=\infty$ and $Y_{D+1-m}=Q$. 

For every $z\ge y\ge Q$, Lemma \ref{lem:tau_m*f}(a) $L_y(1+1/\log z,g)\ll1$. Together with Lemma \ref{lem:L(s,f)}, this implies that
\[
\sum_{y<p\le z} \frac{\Re(g(p))}{p} \le O(1) .
\]
If $m=D$, we have $\Re(g(p)) = \Re(f(p))+D\ge0$, and thus Theorem \ref{thm:main theorem} follows. 

Let us now assume that $m<D$. Let $y\ge Q$ and let $Y=\max\{y,y^{1/|L_y(1,g)|}\}$. For every $z\ge y$, we have that
\[
L_y(1+1/\log z,g) - L_y(1,g) = \int_1^{1+1/\log z} L_y'(\sigma,g) \dee\sigma 
	\ll \frac{\log y}{\log z}
\]
by Lemma \ref{lem:tau_m*f}. Thus, if $z\ge Y$, then $z\ge y^{1/|L_y(1,g)|}$ and thus  $L_y(1+1/\log z,g)\ll |L_y(1,g)|$. Combining this with Lemma \ref{lem:tau_m*f}(b), we conclude that $|L_y(1+1/\log z,g)|\asymp |L_y(1,g)|$ for all $z\ge Y$. Hence, if $u\ge v\ge Y$, then $|L_y(1+1/\log u,g)|\asymp|L_y(1+1/\log v,g)|$. We then use Lemma \ref{lem:L(s,f)} to conclude that
\begin{equation}
	\label{eq:g small for large p}
\sum_{u<p\le v} \frac{\Re(g(p))}{p} = O(1) \quad\text{for all}\ u\ge v\ge Y. 
\end{equation}

Motivated by the above estimate, we set
\[
Y_1 = \inf_{y\ge Q} y^{\max\{1,1/|L_y(1,g)|\}} .
\]
 we have $Q\le Y_1<\infty$, with the rightmost inequality following by our assumption that  $L(s,f)$ has a zero of multiplicity $m$ at $s=1$, which implies that $L(1,g)\neq0$. Moreover, we claim that
\begin{equation}
	\label{Ym-e2}
	|L_y(1+1/\log Y_1,1*g)|  \le O(1) \quad(Q\le y\le Y_1)  .
\end{equation}

To prove \eqref{Ym-e2}, let us first note that the definition of $Y_1$ and \eqref{eq:g small for large p} imply that
\begin{equation}
	\label{eq:Y1}
\sum_{u<p\le v} \frac{\Re(g(p))}{p} =O(1) \quad\text{for all}\ u\ge v\ge Y_1. 
\end{equation}
In addition, we have that
\begin{equation}
	\label{eq:Y_1 asymp}
	\log Y_1 \asymp \min_{y\ge Q} \frac{\log y}{|L_y(1,g)|} ,
\end{equation}
because $L_y(1,g)\ll 1$ for all $y\ge Q$ (see Lemma \ref{lem:tau_m*f}). Thus, if $y\in[Q,Y_1]$, then
\[
\frac{\log Y_1}{\log y} \cdot |L_y(1,g)| \ll1.
\]
On the other hand,
\[
|L_y(1,g)| \asymp |L_y(1+1/\log Y_1,g)| 
\]
by \eqref{eq:Y1} and Lemma \ref{lem:L(s,f)}. Putting the above estimates together proves \eqref{Ym-e2}.

We construct the other numbers $Y_2,Y_3,\dots,Y_{D-m}$ inductively: fix some $k\in\{1,\dots,D-m\}$ and assume that we have constructed $Y_k\le Y_{k-1}\le \cdots \le Y_1$ such that
\begin{equation}
\label{eq:Yk-e1}
\sum_{u<p\le v} \frac{j+\Re(g(p))}{p} = O(1) \quad(Y_{j+1}\le u<v<Y_j)
\end{equation}
for $j\in\{0,1,\dots,k-1\}$, and
\begin{equation}
	\label{eq:Yk-e2}
	|L_y(1+1/\log Y_k,\tau_k*g)|  \le O(1) \quad(Q\le y\le Y_k) .
\end{equation}
We then set 
\[
Z_{k+1}:= \min_{Q\le y\le Y_k} y^{\max\{1,1/|L_y(1+1/\log Y_k,\tau_k*g)|\}} 
\quad\text{and}\quad
Y_{k+1}=\begin{cases}
	Z_{k+1}  &\text{if}\ Z_{k+1}< Y_k,\\
	Y_k  &\text{if}\  Z_{k+1}\ge Y_k.
	\end{cases}
\]
We will prove that \eqref{eq:Yk-e1} holds with $j=k$, and that if $k<D-m$, then \eqref{eq:Yk-e2} holds with $k$ replaced by $k+1$. Before we begin, let us observe that  the definition of $Z_{k+1}$ and \eqref{eq:Yk-e2} imply that 
\begin{equation}
	\label{eq:Z_{k+1} asymp}
\log Z_{k+1}\asymp \min_{Q\le y\le Y_k} \frac{\log y}{|L_y(1+1/\log Y_k,\tau_k*g)|} . 
\end{equation}

\smallskip

Now, the inductive step is easy to establish when $Z_{k+1}\ge Y_k$. Indeed, in this case, $Y_{k+1}=Y_k$, so \eqref{eq:Yk-e1} is vacuous. Moreover, we have 
\[
\log Y_k\le \log Z_{k+1} \ll \frac{\log y}{|L_y(1+1/\log Y_k,\tau_k*f)|} \quad(Q\le y\le Y_k=Y_{k+1}) 
\]
by \eqref{eq:Z_{k+1} asymp}. Together with Lemma \ref{lem:L(s,f)}, this implies that
\[
|L_y(1+1/\log Y_{k+1},\tau_{k+1}*f)|\ll 1 \quad (Q\le y\le Q_{k+1}),
\]
which proves \eqref{eq:Yk-e2} with $k+1$ in place of $k$.

\smallskip 

Let us now assume that $Z_{k+1}<Y_k$, so that $Y_{k+1}=Z_{k+1}$. We start by proving \eqref{eq:Yk-e1} with $k$ replaced by $k+1$. 

\smallskip

First, we consider the case when $k=D-m$. We apply \eqref{eq:Yk-e2}, take logarithms and use Lemma \ref{lem:L(s,f)} to find that
\[
\sum_{y<p\le Y_{D-m}} \frac{D-m+\Re(g(p))}{p} \le O(1) \quad(Q\le y\le Y_{D-m}) .
\]
Since $D-m+\Re(g(p))=D+\Re(f(p))\ge0$ and $Y_{D-m+1}=Q$, relation \eqref{eq:Yk-e1} is proven with $j=D-m=k$. Note that we don't need to establish \eqref{eq:Yk-e2} with $k+1$ in this case (though the reader can check it is also true).

\smallskip

Next, we prove \eqref{eq:Yk-e1} when $j=k<D-m$. Recall \eqref{eq:Z_{k+1} asymp}. Hence, there exists  $y\in[Q,Y_k]$ such that 
\begin{equation}
	\label{eq:Y_{k+1} asymp}
\log Y_{k+1} \asymp \frac{\log y}{|L_y(1+1/\log Y_k,\tau_k*g)|} .
\end{equation}
For all $w\in[y,Y_k]$, we have
\[
	L_y(1+1/\log w,\tau_k*g) - L_y(1+1/\log Y_k,\tau_k*g)  
	= \int_{1+1/\log Y_k}^{1+1/\log w} L_y'(\sigma,\tau_k*g) \dee \sigma .
\]
Using Lemma \ref{lem:tau_m*f}(c) with $z=Y_k$ and with $k_{\text{Lemma}\ \ref{lem:tau_m*f}}=k+m$ , we find that
\begin{equation}
	\label{eq:rate of change}
\big|L_y(1+1/\log w,\tau_k*g) - L_y(1+1/\log Y_k,\tau_k*g) \big| 
\ll  |L_y(1+1/\log Y_k,\tau_k*g)| + \frac{\log y}{\log w} 
\end{equation}
If we further assume that $w\ge Y_{k+1}$, then $\log y/\log w\ll |L_y(1+1/\log Y_k,\tau_k*g)|$ by  \eqref{eq:Y_{k+1} asymp}. We thus conclude that
\[
L_y(1+1/\log w,\tau_k*g) \ll |L_y(1+1/\log Y_k,\tau_k*g)| \quad\text{for all}\ w\in[Y_{k+1},Y_k]. 
\]
Together with Lemma \ref{lem:L(s,f)}, this implies that 
\[
\sum_{w<p\le Y_k} \frac{k+\Re(g(p))}{p} \ge -O(1) \quad\text{for all}\ w\in[Y_{k+1},Y_k].
\]
On the other hand, \eqref{eq:Yk-e2} and Lemma \ref{lem:L(s,f)} imply that 
\[
\sum_{w<p\le Y_k} \frac{k+\Re(g(p))}{p} \le O(1) \quad\text{for all}\ w\in[Y_{k+1},Y_k].
\]
The two above estimates readily establish  \eqref{eq:Yk-e1} with $j=k$. 

Finally, it remains to prove that \eqref{eq:Yk-e2} holds with $k$ replaced by $k+1$. For each $y\in[Q,Y_{k+1}]$, the definition of $Y_{k+1}=Z_{k+1}$ and \eqref{eq:Z_{k+1} asymp} imply that
\[
\frac{\log Y_{k+1}}{\log y} \cdot |L_y(1+1/\log Y_k,\tau_k*g)| \ll 1.
\]
On the other hand,
\[
|L_y(1+1/\log Y_k,\tau_k*g)| \asymp |L_y(1+1/\log Y_{k+1},\tau_k*f)| 
\]
by Lemma \ref{lem:L(s,f)} and by \eqref{eq:Yk-e1} with $j=k$, which we have already established. Consequently, 
\[
|L_y(1+1/\log Y_{k+1},\tau_{k+1}*f)| \asymp \frac{\log Y_{k+1}}{\log y} \cdot |L_y(1+1/\log Y_{k+1}, \tau_k*f)| 
	 \ll 1 
\]
for $y\in[Q,Y_{k+1}]$. This completes the inductive step, thus establishing Theorem \ref{thm:main theorem}.

\section{Zeroes}

In this section, we prove Theorem \ref{thm:zeroes}. Recall that $\CF_D^{\mathrm{strong}}(Q)$ denotes the class of multiplicative functions $f:\N\to\C$ such that $|\Lambda_f|\le D\Lambda$ and 
\[
\bigg|\sum_{n\le x}f(n)\bigg| \le \frac{x^{1-1/\log Q}}{(\log x)^{D+1} } \quad(x\ge Q).
\]
First, we demonstrate the following result that is a relatively standard consequence of the Borel--Carath\'eodory theorem. 

\begin{lemma}
	\label{lem:BC}
	Let $f\in\CF_D^{\mathrm{strong}}(Q)$. 
	There exists a constant $c_0=c_0(D)\in(0,1/10]$ such that $L(s,f)$ has at most $D$ zeroes in the ball $B(1,c_0/\log Q)$ counted according to their multiplicity. Moreover,
		\[
		\frac{L'}{L}(s,f) = \sum_{\rho \in B(1,c_0/\log Q)} \frac{1}{s-\rho}  +O_D(\log Q) 
		\]
		for all $s\in B(1,c_0/(2\log Q))$, where the sum runs over all zeros $\rho\in B(1,c_0/\log Q)$ of $L(s,f)$ listed according to their multiplicity.
\end{lemma}

\begin{proof} As usually, we write $s=\sigma+it$. Let also $\alpha$ be as in Lemma \ref{lem:sifted-f}, and let $\delta=\alpha/\log Q$. By Lemma \ref{lem:sifted-f} and partial summation, we have that
	\[
	L_Q(s,f)\ll_D 1     \quad (\sigma\ge 1-\delta,\ |t|\le 1). 
	\]
Also, $L_Q(1+\delta/100,f)\gg_D1$. We then apply the Borel-Carath\'eodory theorem in the form of Lemma 8.6 in \cite{dk-book} with $R=\delta/4$, $M\asymp_D 1$ and $g(z)=L_Q(z+1+\delta/100,f)$ to find that
\[
\frac{L_Q'}{L_Q}(s,f) = \sum_{\rho \in B(1+\delta/100, \delta/2)}  \frac{1}{s-\rho} + O_D(1/\delta) .
\]
for all $s\in B(1+\delta/100,\delta/4)$, where the summation runs over all zeroes $\rho$ of $L(s,f)$ in the ball $B(1+\delta/100,\delta/2)$ according to their multiplicity. Note also that $\sum_{p\le Q}\sum_{k\ge1} \frac{\log p}{p^{k(1-\delta)}} \ll_D 1/\delta $. Thus
\begin{equation}
	\label{eq:BC}
	\frac{L'}{L}(s,f) = \sum_{\rho \in B(1+\delta/100, \delta/2)}  \frac{1}{s-\rho} + O_D(1/\delta)  
\end{equation}
for all $s\in B(1+\delta/100,\delta/4)$.

Now, let $M\ge100$ be a constant to be chosen later and let $s_1=1+\delta/M$. 
If $\rho\in B(1,\delta/M^2)$, then
\begin{equation}
	\label{eq:rho contribution}
	\Re\bigg(\frac{1}{s_1-\rho}\bigg) = \frac{s_1-\beta}{|s_1-\rho|^2} = \frac{M+O(1)}{\delta} .
\end{equation}
For all other zeroes $\rho\in B(1+\delta/100,\delta/2)$, we note that $\Re(1/(s_2-\rho))\ge0$.  Consequently,
\begin{equation}
	\label{eq:rho total ower bound}
\Re\sum_{\rho\in B(1+\delta/100,\delta)} \frac{1}{s-\rho}\ge  \#\big\{\rho\in  B(1,\delta/M^2)    \big\} \cdot \frac{M+O(1)}{\delta} .
\end{equation}
On the other hand, since $|\Lambda_f|\le D\Lambda$, we have 
\begin{equation}
	\label{eq:L'/L bound}
	\bigg|\frac{L'}{L}(s_1,f)\bigg| \le -D \frac{\zeta'}{\zeta}(s_1) = \frac{DM}{\delta}+O_D(1) .
\end{equation}
Inserting \eqref{eq:rho total ower bound} and \eqref{eq:L'/L bound} into \eqref{eq:BC} implies that 
\[
\#\big\{\rho\in  B(1,\delta/M^2)   \big\} \cdot \frac{M+O(1)}{\delta}   \le \frac{DM}{\delta} + O_D(1/\delta) .
\]
Taking $M$ large enough and $c_0=\alpha/M^2$ proves the first part of the lemma. 

To prove the second part, we note that \eqref{eq:BC} readily yields that
\[
\frac{L'}{L}(s,f) = \sum_{\rho \in B(1,c_0/\log Q)}  \frac{1}{s-\rho} + O_D\bigg(\frac{1+\#\big\{\rho\in B(1+\delta/100,\delta/2)\setminus B(1,c_0/\log Q)\big\} }{\delta}\bigg)  
\]
for all $s\in B(1,c_0/(2\log Q))$, because $|s-\rho|\ge c_0/(2\log Q)\asymp_D\delta$ for all such $s$ and for all $\rho$ such that $|\rho-1|>c_0/\log Q$. 
To complete the proof, it suffices to show that 
\begin{equation}
	\label{eq:zeros crude bound}
	\#\big\{\rho \in B(1+\delta/100,\delta/2)\big\} \ll_D 1. 
\end{equation}
This follows by a variation of the argument leading to our bound for the number of zeroes in $B(1,\delta/M^2)$. Indeed, if we let $s_2=1+\delta/4$, then the equality in \eqref{eq:rho contribution} implies that 
\[
\Re\bigg(\frac{1}{s_2-\rho}\bigg) \ge \frac{4}{9\delta} 
\]
for all $\rho \in B(1+\delta/100,\delta/2)$, because $s_2-\beta\ge s_2-1= \delta/4$ and $|s_2-\rho|\le |s_2-1-\delta/100|+|1+\delta/100-\rho|\le 3\delta/4$. Moreover, arguing as in \eqref{eq:L'/L bound}, we find that $(L'/L)(s_2,f)\ll_D1/\delta$. Combining these inequalities with \eqref{eq:BC} demonstrates \eqref{eq:zeros crude bound}, this completing the proof of the lemma.
\end{proof}

\begin{proof}[Proof of Theorem \ref{thm:zeroes}]
	The first claim in the statement of the theorem follows immediately from Lemma \ref{lem:BC}. Now, let $m$ and $\rho_m,\rho_{m+1},\dots,\rho_d,\rho_{d+1}$ be as in the statement of Theorem \ref{thm:zeroes}. Fix $j\in\{m,m+1,\dots,d\}$ and let $z\ge y\ge Q^{2/c_0}$ so that
	\begin{equation}
		\label{eq:y and z in terms of zeroes}
	\frac{2}{|\rho_{j+1}-1|} \le \log y\le \log z<\frac{1}{2|\rho_j-1|}. 
		\end{equation}
To complete the proof of Theorem \ref{thm:zeroes}, it suffices to show that 	
		\begin{equation}
			\label{eq:thm zeroes: goal}
				 \sum_{y<p\le z} \frac{\Re(f(p))+j}{p} = O_D(1)
		\end{equation}
		uniformly over all such choices of $y$ and $z$. 
		
		Indeed, using Lemma \ref{lem:L(s,f)} twice, we have that
\begin{align*}
\sum_{y<p\le z} \frac{\Re(f(p))}{p} 
	&= \log|L(1+1/\log z,f)| - \log|L(1+1/\log y,f)| + O_D(1) \\
	&= - \Re \int_{1+1/\log z}^{1+1/\log y} \frac{L'}{L}(s,f) \dee s + O_D(1). 
\end{align*}
We estimate the integrant using Lemma \ref{lem:BC} to deduce that
\begin{align*}
	\sum_{y<p\le z} \frac{\Re(f(p))}{p} 
	&= - \sum_{\rho\in B(1,c_0/\log Q)} \Re \int_{1+1/\log z}^{1+1/\log y} \frac{\dee s}{s-\rho} +O_D(1) \\
	&=- \sum_{\rho\in B(1,c_0/\log Q)}  \log\bigg|\frac{\frac{1}{\log y}-(\rho-1)}{\frac{1}{\log z}-(\rho-1)} \bigg| + O_D(1).
\end{align*}
There are $m$ summands with $\rho=1$. Moreover, there are $j-m$ summands with $0<|\rho-1|\le |\rho_j-1|$. Since $|\rho_j-1|\le 1/(2\log z)$ by \eqref{eq:y and z in terms of zeroes},  we have that $|1/\log y - (\rho-1)|\asymp 1/\log y$ and $|1/\log z-(\rho-1)|\asymp1/\log z$ for all such $\rho$, so that
\[
\bigg|\frac{\frac{1}{\log y}-\rho}{\frac{1}{\log z}-\rho} \bigg| \asymp \frac{\log z}{\log y} .
\]
For all remaining zeroes $\rho$, we have $|\rho-1|\ge |\rho_{j+1}-1|\ge 2/\log y$ by \eqref{eq:y and z in terms of zeroes}, and thus $|1/\log y - (\rho-1)|\asymp |\rho-1|$ and $|1/\log z-(\rho-1)|\asymp|\rho-1|$, so that
\[
\bigg|\frac{\frac{1}{\log y}-\rho}{\frac{1}{\log z}-\rho} \bigg| \asymp 1 .
\]
Putting together the above estimates completes the proof of Theorem \ref{thm:zeroes}.
\end{proof}

\subsection*{Acknowledgments}

The author would like to thank Andrew Granville, Youness Lamzouri and Jesse Thorner for their comments on an earlier version of the paper.

The author gratefully acknowledges support by the Courtois Chair II in fundamental research, by the Natural Sciences and Engineering Research Council of Canada (RGPIN-2024-05850), by the Fonds de recherche du Qu\'ebec - Nature et technologies (2025-PR-345672), and by the program {\it Simons Fellows in Mathematics} of the Simons Foundation.

\bibliographystyle{alpha}

\end{document}